\documentclass[12pt]{article}
\usepackage{amsmath,amssymb,amsthm}
\date{}
\newcommand{\R}{{\mathbb R}}

\newcommand{\N}{{\mathbb N}}
\newcommand{\C}{{\mathbb C}}

\newcommand{\D}{{\mathcal D}}

\newcommand{\E}{{\mathcal E}}

\newcommand{\const}{{\rm const}}
\newcommand{\esslim}{\mathop{\rm ess\,lim}}

\newcommand{\supp}{\mathop{\rm supp}}
\renewcommand{\div}{{\rm div}}

\renewcommand{\>}{\rangle}
\newcommand{\codim}{\mathop{\rm codim}}
\renewcommand{\Im}{\mathop{\rm Im}}
\renewcommand{\S}{\mathcal{S}}
\numberwithin{equation}{section}

\theoremstyle{plain}
\newtheorem{theorem}{Theorem}[section]
\newtheorem{lemma}{Lemma}[section]
\newtheorem{proposition}{Proposition}[section]
\newtheorem{corollary}{Corollary}[section]

\theoremstyle{definition}
\newtheorem{definition}{Definition}[section]
\newtheorem{remark}{Remark}[section]

\title{On one criterion of the uniqueness of generalized solutions for linear transport equations with discontinuous coefficients}
\author{E.Yu.~Panov\footnote{Novgorod State University, e-mail: \textbf{Eugeny.Panov@novsu.ru}}}

\begin{document}

\maketitle
\begin{abstract}
We study generalized solutions of multidimensional transport equation with bounded measurable solenoidal field of coefficients $a(x)$. It is shown that any generalized solution satisfies the renormalization property if and only if the operator $a\cdot\nabla u$, $u\in C_0^1(\R^n)$ in the Hilbert space $L^2(\R^n)$ is an essentially skew-adjoint operator, and this is equivalent to the uniqueness of generalized solutions. We also establish existence of a contractive semigroup, which provides generalized solutions, and give a criterion of its uniqueness.
\end{abstract}

\section{Introduction}\label{sec1}

We study the following evolutionary linear transport equation
\begin{equation}
\label{1}
u_t+\sum_{i=1}^n a_i(x)u_{x_i}=0,
\end{equation}
where $u=u(t,x)$, $(t,x)\in\Pi=(0,+\infty)\times\R^n$.

In the case when the field of coefficients $a=(a_1(x),\ldots,a_n(x))\in C^1(\R^n,\R^n)$ the
theory of solutions (both classical and generalized) to the Cauchy problem
for equation (\ref{1}) is well-known and it is covered by the method
of characteristics. The case when the coefficients are generally discontinuous
is more interesting and more complicated. The well-posedness of Cauchy problem for such equations is established  under some additional restrictions on
coefficients. Some results in this direction could be found in papers
\cite{DiL,Amb}. The equations like (\ref{1}) with general solenoidal vector
of coefficients naturally arise in the study of some important nonlinear
conservation laws (~see for instance, \cite{ABL}~). The solenoidality condition $\div a(x)=0$ (in distributional sense)
allows to rewrite the equation in divergence form
$$
u_t+\div_x (a(x)u)=0
$$
and introduce generalized
solutions (g.s.) of the corresponding Cauchy problem with initial data
\begin{equation}\label{2}
u(0,x)=u_0(x).
\end{equation}
The coefficients $a_i$, $i=1,\ldots,n$ are supposed to be bounded: $a(x)\in L^\infty(\R^n,\R^n)$.
We denote $\bar\Pi=[0,+\infty)\times\R^n$.

\begin{definition}\label{def1}
A function $u=u(t,x)\in L^1_{loc}(\bar\Pi)$
is called a g.s. of the problem
(\ref{1}), (\ref{2}) if for all $f=f(t,x)\in C_0^\infty(\bar\Pi)$
\begin{equation}\label{3}
\int_{\Pi} [uf_t+au\cdot\nabla_x f]dtdx+\int_{\R^n} u_0(x)f(0,x)dx=0.
\end{equation}
Here and below we use the notation $\cdot$ for the scalar multiplication on $\R^n$.
\end{definition}

Taking in (\ref{3}) test functions $f\in  C_0^\infty(\Pi)$, we derive that
\begin{equation}\label{4}
u_t+\div_x (a(x)u)=0
\end{equation}
in the sense of distributions on $\Pi$ (~in $\D'(\Pi)$~). Besides, (\ref{3}) readily implies that
\begin{equation}\label{5}
\esslim_{t\to 0} u(t,\cdot)=u_0 \ \mbox{ in } \D'(\R^n).
\end{equation}
Actually, (\ref{3}) is equivalent to (\ref{4}), (\ref{5}). For the details see \cite[Proposition 2]{PaTr}.

For classical solutions $u(t,x)\in C^1(\bar\Pi)$ of transport equations (\ref{1}), it is clear that compositions
$g(u)$ remain to be solutions for every $g(u)\in C^1(\R)$. This fact, called the renormalization property, is readily follows from the chain rule. For generalized solutions the renormalization property may fail (~cf. \cite{Aiz,CLR}~). This induce us to introduce the specific notion of \textit{a renormalized solution}.

\begin{definition}\label{def2}
A function $u=u(t,x)\in L^1_{loc}(\bar\Pi)$
is called a renormalized solution of the problem (\ref{1}), (\ref{2}) if for any $g(u)\in C(\R)$ such that $g(u_0(x))\in L^1_{loc}(\R^n)$, $g(u(t,x))\in L^1_{loc}(\bar\Pi)$ the function $g(u(t,x))$ is a g.s. of problem (\ref{1}), (\ref{2}) with initial data $g(u_0(x))$.
\end{definition}

We need the following simple a-priory estimate for nonnegative g.s. (below we denote by $|x|$ the Euclidean norm of a finite-dimensional vector $x$ ).

\begin{proposition}\label{pro1}
Let $u=u(t,x)\ge 0$ be a g.s. of the problem
{\rm (\ref{1}), (\ref{2}) }. Then for a.e. $t>0$ for each $R>0$
\begin{eqnarray}
\label{6}
\int_{|x|<R}u(t,x)dx \le\int_{|x|<R+Nt}u_0(x)dx, \\
\label{6a}
\int_{|x|>R+Nt}u(t,x)dx\le\int_{|x|>R}u_0(x)dx,
\end{eqnarray}
where $N=\|a\|_\infty$.
\end{proposition}

\begin{proof}
Choose a function
$\beta(s)\in C_0^\infty(\R)$ such that
$\supp\beta(s)\subset [0,1]$, $\beta(s)\ge 0$, and
$\displaystyle\int\beta(s)ds=1$ and set for
${\nu\in\N}$  \ $\beta_\nu(s)=\nu\beta(\nu s)$,
$\displaystyle\theta_\nu(t)=\int_{-\infty}^t \beta_\nu(s)ds$. It is clear that
$\beta_\nu(s)\in C_0^\infty(\R)$,
$\supp\beta_\nu(s)\subset [0,1/\nu]$, $\beta_\nu(s)\ge 0$,
$\displaystyle\int\beta_\nu(s)ds=1$. Therefore, the sequence
$\beta_\nu(s)$ converges to Dirac
$\delta$-function in $\D'(\R)$ as $\nu\to\infty$, and the sequence
$\theta_\nu(t)$ is bounded
(~$0\le\theta_\nu(t)\le 1$~) and converges pointwise to the Heaviside function
$\displaystyle\theta(t)=\left\{\begin{array}{ll} 0, & t\le 0, \\ 1, & t>0.
\end{array}\right. $
Let $p=p(s)\in C^\infty(\R)$, $p'(s)\ge 0$, $p(s)=0$ for $s\le -1$,
$p(s)=1$ for $s\ge 0$. Set for $t_0>0$, $r>Nt_0$,
$\nu\in\N$ \ $f=f(t,x)=p(r-Nt-|x|)\theta_\nu(t_0-t)$.
Then $f\in C_0^\infty(\bar\Pi)$
and by identity (\ref{3})
\begin{eqnarray}
\label{7}
\theta_\nu(t_0)\int_{\R^n} u_0(x)p(r-|x|)dx-\int_0^{+\infty}\!\!
\int_{\R^n}u(t,x)p(r-Nt-|x|)dx
\delta_\nu(t_0-t)dt
\nonumber\\
-\int_\Pi\![N\!+\!a(x)\cdot x/|x|]
p'(r-Nt-|x|)u(t,x)\theta_\nu(t_0-t)dtdx\!=\!0.
\end{eqnarray}
Since $|a(x)\cdot x/|x||\le |a(x)|\le N$ and $p'(s)\ge 0$, the last integral in
(\ref{7}) is nonnegative. Therefore, (\ref{7}) implies the inequality
\begin{equation}
\label{8}
\int_0^{+\infty}\!\int_{\R^n}u(t,x)p(r-Nt-|x|)dx\delta_\nu(t_0-t)dt\le
\theta_\nu(t_0)\int_{\R^n}u_0(x)p(r-|x|)dx.
\end{equation}
Let
$\E\subset\R_+$ be the set of full measure consisting of values $t>0$ such that
$u(t,x)\in L^1_{loc}(\R^n)$ and $t$ is a Lebesgue point
of functions
$\displaystyle F_r(t)=\int_{\R^n}u(t,x)p(r-Nt-|x|)dx$
for all rational $r$. Since $F_r(t)$ depends continuously on the parameter
$r$ then $t\in\E$ is a Lebesgue point of $F_r(t)$ for all real $r$.
Let $t_0\in\E$. Passing to the limit in (\ref{8}) as $\nu\to\infty$,
we obtain that
$$
F_r(t_0)\le F_r(0)=\int_{\R^n}u_0(x)p(r-|x|)dx.
$$
Thus  $\forall t=t_0\in\E$, $r>Nt$
\begin{equation}
\label{9}
\int_{\R^n}u(t,x)p(r-Nt-|x|)dx\le
\int_{\R^n}u_0(x)p(r-|x|)dx.
\end{equation}
Obviously, the set $\E$ of full measure could be chosen common for a countable
family of functions $p=p_k(s)$, approximating the Heaviside function. Taking $p=p_k$ in (\ref{9})
and passing to the limit as $k\to\infty$, we conclude that
$\forall t\in\E$, $r>Nt$
$$
\int_{|x|<r-Nt}u(t,x)dx\le
\int_{|x|<r}u_0(x)dx
$$
and to complete the proof of (\ref{6}) it only remains to substitute
$r=R+Nt$ in the obtained inequality.

Similarly, to establish (\ref{6a}) we choose the test function $f=f(t,x)=\chi(t,x)\theta_\nu(t_0-t)\in C^\infty(\bar\Pi)$, where $\chi(t,x)=(p(R-Nt-|x|)-p(r+Nt-|x|))$, $R>r>0$, $R>Nt_0$.
By (\ref{3}) we obtain
\begin{eqnarray}
\label{7a}
\theta_\nu(t_0)\int_{\R^n} u_0(x)\chi(0,x)dx-\int_0^{+\infty}\!\!
\int_{\R^n}u(t,x)\chi(t,x)dx
\delta_\nu(t_0-t)dt+
\nonumber\\
\int_\Pi\!u(t,x)[\chi_t\!+\!a(x)\cdot\nabla_x\chi]
\theta_\nu(t_0-t)dtdx\!=\!0.
\end{eqnarray}
Since $\chi_t=-N((p'(R-Nt-|x|)+p'(r+Nt-|x|))\le 0$ while
\begin{eqnarray*}
|a(x)\cdot\nabla_x\chi|\le |a(x)||\nabla_x\chi|\le N|p'(R-Nt-|x|)-p'(r+Nt-|x|)|\le\\ N(p'(R-Nt-|x|)+p'(r+Nt-|x|)),
\end{eqnarray*}
we see that the last integral in
(\ref{7a}) is nonpositive and from (\ref{7}) it follows that
\begin{eqnarray}
\label{8a}
\int_0^{+\infty}\!\int_{\R^n}u(t,x)(p(R-Nt-|x|)-p(r+Nt-|x|))dx\delta_\nu(t_0-t))dt\le\nonumber\\
\theta_\nu(t_0)\int_{\R^n}u_0(x)(p(R-|x|)-p(r-|x|))dx.
\end{eqnarray}
Obviously, the set $\E_1$ of common Lebesgue points of all functions of the kind
$$F(t)=\int_{\R^n}u(t,x)(p(R-Nt-|x|)-p(r+Nt-|x|))dx$$ has full Lebesgue measure. Assuming that $t_0\in\E_1$ and passing to the limit as $\nu\to\infty$, we arrive at the inequality
\begin{eqnarray*}
\int_{\R^n}u(t_0,x)(p(R-Nt_0-|x|)-p(r+Nt_0-|x|))dx\le\\ \int_{\R^n}u_0(x)(p(R-|x|)-p(r-|x|))dx.
\end{eqnarray*}
Taking in this estimate $p=p_k(s)$, $k\in\N$ (recall that this sequence converges to the Heaviside function) and passing to the limit as $k\to\infty$, we obtain that for all $t\in\E_1$
$$
\int_{r+Nt<|x|<R+Nt}u(t,x)dx\le\int_{r<|x|<R}u_0(x)dx.
$$
To complete the proof, we pass to the limit in this inequality as $R\to\infty$ and replace $r$ by $R$.
\end{proof}

Let us introduce the linear operator $A_0=\div (au)=a(x)\cdot\nabla u(x)$ in the real Hilbert space $L^2=L^2(\R^n)$.
This operator is defined on a dense subspace $C_0^1(\R^n)\subset L^2$. For every $u,v\in C_0^1(\R^n)$
\begin{eqnarray*}
(Au,v)_2=\int_{\R^n} (a(x)\cdot\nabla u(x))v(x)dx=-\int_{\R^n}u(x)a(x)\cdot\nabla v(x)dx+ \\
\int_{\R^n} a(x)\cdot\nabla(u(x)v(x))dx=-\int_{\R^n}u(x)a(x)\cdot\nabla v(x)dx=-(u,Av)_2,
\end{eqnarray*}
where we use the fact that $\div a=0$ in $\D'(\R^n)$. Here we denote by $(f,g)_2$ the scalar multiplication in $L^2$: ${(f,g)_2=\int_{\R^n}f(x)g(x)dx}$.

The obtained identity means that $A_0$ is skew-symmetric operator. Therefore, it admits the closure, which we define by $A$. $A$ is a closed skew-symmetric operator: $-A\subset A^*$. It is easy to see that the conjugate operator is defined as follows $v=A^*u$ if and only if $u,v\in L^2$ and $-\div(au)=v$ in $\D'(\R^n)$.

Our main results are the following criteria.

\begin{theorem}\label{main}
(i) The necessary and sufficient condition for any g.s. $u(t,x)\in L^2_{loc}(\bar\Pi)$ to be a renormalized solution of (\ref{1}), (\ref{2}) (with $u_0\in L^2_{loc}(\R^n)$) is that the operator $A$ is skew-adjoint; (ii) The same condition is necessary and sufficient for the uniqueness of any g.s. $u(t,x)\in L^2_{loc}(\bar\Pi)$.
\end{theorem}

In Theorem~\ref{th5} below we also give a necessary and sufficient condition of uniqueness of contraction semigroups on $L^2(\R^n)$, which provide g.s.

\begin{remark}\label{rem}
Let us consider the Banach space
$$ 
X=D(A^*)=\{ \ u\in L^2 \ | \ A^*u=-\div (au)\in L^2 \ \} 
$$
equipped with the graph norm $\|u\|=\|u\|_2+\|A^*u\|_2$. It is clear that the operator $A$ is skew adjoint id and only if the space $C_0^1(\R^n)$ is dense in $X$. This condition is similar to the criterion of the uniqueness for  both the forward and the backward Cauchy problems, suggested in \cite[Theorem~2.1]{BoCr}.
\end{remark}

\section{The case of smooth coefficients}\label{sec2}

In the case when the coefficients $a_i(x)\in C^1(\R^n)\cap L^\infty(\R^n)$, $i=1,\ldots,n$, are smooth the existence and uniqueness of g.s. is well known.
In this case a g.s. of the problem
(\ref{1}), (\ref{2}) can be found by the method of characteristics, see \cite[Proposition 3]{PaTr} for details. The characteristics of equation (\ref{1}) are integral curves
$(t,x(t))$ of the system of ordinary differential equations
\begin{equation}
\label{10}
\dot x=a(x),
\end{equation}
and they are defined for all $t\in\R$ since the right-hand side of (\ref{10})
is bounded. For $(t_0,x_0)\in\Pi$ we denote
by $x(t;t_0,x_0)$ the solution of (\ref{10}) such that
$x(t_0)=x_0$, we also denote $y(t_0,x_0)=x(0;t_0,x_0)$ (~i.e., the source of characteristic $x(t;t_0,x_0)$~).
Then any g.s. $u(t,x)$ of the problem (\ref{1}), (\ref{2}) should be constant on
characteristics (possibly after correction on a set of null Lebesgue measure), which implies that
$u(t,x)=u_0(y(t,x))$.
We observe that the  map
$(t,x)\to (t,y(t,x))$ is a diffeomorphism on $\Pi$, which implies that
$u(t,x)$ is measurable and the correspondence
$u_0\to u$ keeps the relation of equality almost everythere.
Besides, in view of the solenoidality assumption for each $t\in\R$ the map $x\to y(t,x)$ conserves the Lebesgue measure. This readily implies that for all $t\in\R$
\begin{equation}\label{11}
\int_{\R^n} u(t,x)dx=\int u_0(x)dx
\end{equation}
whenever these integrals exist.
The above observations allow to obtain the following properties of g.s.

\begin{proposition}\label{pro2} Assume that $u(t,x)=u_0(y(t,x))$ be the unique g.s. of problem (\ref{1}), (\ref{2}) (defined for all real times $t$). Then

(i) For every continuous function $g(u)$ such that $g(u_0)\in L^1_{loc}(\R^n)$ the composition $g(u(t,x))$ is a g.s. of (\ref{1}), (\ref{2}) with initial function $g(u_0(x))$ (renormalization property);

(ii) If $u_0\le v_0$ almost everywhere (a.e.) on $\R^n$, and $u=u(t,x)$, $v=v(t,x)$ are g.e.s. of (\ref{1}), (\ref{2}) with initial functions $u_0$, $v_0$, respectively, then $u(t,x)\le v(t,x)$ a.e. on $\R^{n+1}$ (monotonicity);

(iii) Let $T_tu=u(y(t,x))$. Then $T_{t+s}u=T_t(T_su)$ (group property);

(iv) If $u_0(x)\in L^p(\R^n)$, where $1\le p\le\infty$, then $u(t,\cdot)\in L^p(\R^n)$ for all $t\in\R$ and
$\|u(t,\cdot)\|_p=\|u_0\|_p$. Moreover, if $p<\infty$, then
\begin{equation}\label{mod}
\|u(t+h,\cdot)-u(t,\cdot)\|_p\le\omega_p(h)\doteq\inf_{v\in C_0^1(\R^n)}\left(2\|u_0-v\|_p+N C(v) |h|\right)\mathop{\to}_{h\to 0} 0,
\end{equation}
where the constant $C(v)$, given below in (\ref{Cv}), depends only on $v$.
In particular, the map $t\to T_tu_0=u(t,\cdot)\in L^p(\R^n)$ is uniformly continuous on $\R$.
\end{proposition}

\begin{proof}
Properties (i), (ii) readily follows from the representations $u=u_0(y(t,x))$, $v=v_0(y(t,x))$.
To prove (iii), notice that $y(t+s,x)=x(0;t+s,x)=x(0;s,x(s;t+s,x))=x(0;s,x(0;t,x))=y(s,y(t,x))$, where we used that
$x(t+h;t_0+h,x_0)\equiv x(t;t_0,x_0)$ $\forall h\in\R$ because characteristic system (\ref{10}) is autonomous.
This readily implies the group property
$$
T_{t+s}u(x)=u(y(t+s,x))=u(y(s,y(t,x)))=(T_su)(y(t,x))=T_t(T_su)(x).
$$

If $u_0\in L^\infty(\R^n)$, the representation $u(t,x)=u_0(y(t,x))$ yields $u(t,\cdot)\in L^\infty(\R^n)$, $\|u(t,\cdot)\|_\infty=\|u_0\|_\infty$. If $p<\infty$, then by assertion (i) with $g(u)=|u|^p$ and identity (\ref{11}) we find
$$
\int_{\R^n} |u(t,x)|^pdx=\int_{\R^n}|u_0(x)|^pdx \ \forall t\in\R,
$$
that is, $u(t,\cdot)\in L^p(\R^n)$, $\|u(t,\cdot)\|_p=\|u_0\|_p$.
Finally, let $u_0\in L^p(\R^n)$, $v=u(t,\cdot)=T_tu_0$. Then by group property (iii) we find
$$
\|u(t+h,\cdot)-u(t,\cdot)\|_p=\|T_t(T_hu_0-u_0)\|_p=\|T_hu_0-u_0\|_p=\|u_0(y(h,x))-u_0(x)\|_p.
$$
We observe that $y(h,x)-x=x(0)-x(h)$, where $x(t)=x(t;h,x)$, and since $\dot x(t)=a(x(t))$, then
$$
|y(h,x)-x|=\left|\int_0^h a(x(t))dt\right|\le\left|\int_0^h |a(x(t))|dt\right|\le N|h|,
$$
$N=\|a\|_\infty$. If $v(x)\in C_0^1(\R^n)$, then
\begin{eqnarray}\label{est}
\|T_hv-v\|_p=\|v(y(h,x))-v(x)\|_p=\left(\int_{A_v\cup A_v^h}|v(y(h,x))-v(x)|^pdx\right)^{1/p}\nonumber \\ \le \|\nabla v\|_\infty\left(\int_{A_v\cup A_v^h} |y(h,x)-x|^pdx\right)^{1/p}\le\nonumber\\
\|\nabla v\|_\infty(m(A_v)+m(A_v^h))^{1/p}N|h|,
\end{eqnarray}
where $A_v,A_v^h$ are subsets of $\R^n$, determined by the relations $v(x)\not=0$, $v(y(h,x))\not=0$, respectively, and by $m(A)$ we denote the Lebesgue measure of a measurable set $A$. Since the map $y(h,\cdot)$ keeps the Lebesgue measure,
$m(A_v^h)=m(y(h,\cdot)^{-1}(A_v))=m(A_v)$ and, in view of (\ref{est}),
$$
\|T_hv-v\|_p\le NC(v)|h|,
$$
where
\begin{equation}\label{Cv}
C(v)=C(v)\|\nabla v\|_\infty(2m(A_v))^{1/p}
\end{equation}
 (notice that, in view of assumption $v\in C_0^1(\R^n)$, the set $A_v$ is bounded and, therefore, $m(A_v)<\infty$).
Therefore, for all $v\in C_0^1(\R^n)$
\begin{eqnarray*}
\|T_hu_0-u_0\|_p\le\|T_hu_0-T_hv\|_p+\|T_hv-v\|_p+ \|v-u_0\|_p =\\ \|T_h v-v\|_p+ 2\|u_0-v\|_p\le 2\|u_0-v\|_p+NC(v)|h|,
\end{eqnarray*}
and (\ref{mod}) follows. Let us show that $\omega_p(h)\to 0$ as $h\to 0$.
For arbitrary $\varepsilon>0$ we can find $v\in C_0^1(\R^n)$ such that $\|u_0-v\|_p\le\varepsilon/2$. Then
$$
\omega_p(h)\le 2\|u_0-v\|_p+N C(v)|h|\le\varepsilon+NC(v)|h|.
$$
Hence,
$$
\limsup_{h\to 0}\omega_p(h)\le\varepsilon
$$
and since $\varepsilon>0$ is arbitrary, we derive that $\lim\limits_{h\to 0}\omega_p(h)=0$.
This completes the proof.
\end{proof}

As follows from assertions (iii), (iv) of Proposition~\ref{pro2}, the linear operators $T_tu_0=u(t,\cdot)=u_0(y(t,x))$ generate the $C_0$-group of linear isomorphisms on $L^p(\R^n)$. In the particular case $p=2$ the operators $T_t$, $t\in\R$ is a group of unitary operators in the Hilbert space $L^2(\R^n)$. Let $\displaystyle Bu=\lim_{t\to 0}\frac{T(t)u-u}{t}$ be the infinitesimal generator of this group. This operator is defined in the domain
$D(B)$ consisting on such $u\in L^2(\R^n)$ that $\displaystyle\lim_{t\to 0}\frac{T(t)u-u}{t}$ exists in $L^2$.
It is known that $D(B)$ is a dense subspace and $B$ is a closed, possibly unbounded, operator. Since $T(t)$ is an unitary group, then by Stone's theorem $B$ is a skew-adjoint operator.
If $u(t,x)=T_tu(x)$, then $u_t=-\div au$ in $\D'(\R^{n+1})$. Hence,
it is natural to expect that $B=-A$, where the operator $A$ was defined above, in the end of Introduction.

\begin{theorem}\label{th1}
The operator $B$ coincides with $-A$. In particular, the operator $A=-B$ is skew-adjoint.
\end{theorem}

\begin{proof}
First, we remark that $-A_0\subset B$. Indeed, if $u(x)\in C_0^1(\R^n)=D(A_0)$, then $u(t,x)=T_tu(x)\in C^1(\R^{n+1})$ is a classic solution of (\ref{1}). Therefore,
$$
\lim_{t\to 0}\frac{T_tu(x)-u(x)}{t}=u_t(0,x)=-a(x)\cdot\nabla u(x)=-A_0u(x).
$$
Obviously, this limit is uniform with respect to $x\in\R^n$, which implies that
$$
\lim_{t\to 0}\frac{T_tu-u}{t}=-A_0u \ \mbox{ in } L^2.
$$
Hence, $u\in D(B)$ and $Bu=-A_0u$. Since $B$ is closed, then also $-A\subset B$ (recall that $A$ is the closure of  operator $A_0$~). In particular, $B=-B^*\subset A^*$. We will show that actually $B=A^*$. Let $u\in D(A^*)$.
Then $f=u+A^*u\in L^2$. Since $B$ is skew-adjoint, the operator $E+B$ is invertible and $(E+B)^{-1}$ is a bounded operator on $L^2$. Let $v=(E+B)^{-1}f\in D(B)$. Then $v+Bv=v+A^*v=f=u+A^*u$, and the function $w=u-v$ satisfies the relation $w-\div aw=0$ in $\D'(\R^n)$. As follows from DiPerna-Lions renormalization lemma \cite[Lemma II.1]{DiL}, $2w^2=2w\div aw=\div aw^2$ in $\D'(\R^n)$. Applying this relation to the test function $\rho(\varepsilon x)$, where $\rho(y)\in C_0^1(\R^n)$, $\rho(y)\ge 0$, $\rho(0)=1$, and $\varepsilon>0$, we arrive at the equality
$$
2\int_{\R^n}w^2\rho(\varepsilon x)dx=-\varepsilon\int_{\R^n}w^2a(x)\cdot\nabla_y\rho(\varepsilon x)dx.
$$
Passing in this equality to the limit as $\varepsilon\to 0$, we deduce that $\|w\|_2=0$. Hence, $u=v\in D(B)$.
We have proven that $D(A^*)=D(B)$. This means that $B=A^*$. This, in turn, implies $B=-B^*=-A^{**}=-A$. The proof is complete.
\end{proof}

\section{Main result: the necessity}\label{sec3}
Now we consider the case of general solenoidal field of coefficients $a=a(x)\in L^\infty(\R^n,\R^n)$.
Let
$$
\gamma_\nu(\xi)=\nu^{n}\prod_{i=1}^n\beta(\nu\xi_i)
$$
be a sequence of averaging kernels (approximate unity), where $\xi\in\R^n$, $\nu\in\N$, and
the function $\beta(s)$ was defined above in the proof of Proposition~\ref{pro1}.
Introduce sequences of averaged coefficients, setting for $x\in\R^n$
\begin{eqnarray*}
a_\nu(x)=(a_{1\nu}(x),\ldots,a_{n\nu}(x))=
a*\gamma_\nu(x)= \int_{\R^n} a(x-\xi)
\gamma_\nu(\xi)d\xi.
\end{eqnarray*}
By the known property of averaging functions, $a_\nu\in C^\infty(\R^n,\R^n)\cap L^\infty(\R^n,\R^n)$, $\|a_\nu\|_\infty\le\|a\|_\infty\doteq N$, and
$\div a_\nu(x)=(\div a)*\gamma_\nu(x)=0$. As was demonstrated in the previous section, there exists a unique g.s.
$u=u_\nu(t,x)$ of the Cauchy problem for the regularized equation
\begin{equation}\label{13}
u_t+a_\nu(x)\cdot\nabla_x u=u_t+\div(a_\nu u)=0
\end{equation}
with initial condition (\ref{2}), which may be considered for all time $t\in\R$. By the renormalization property (i) for any $r\ge 0$ the function
$(|u_\nu(t,x)|-r)^+=\max(|u_\nu(t,x)|-r,0)$ is a g.s. of (\ref{13}), (\ref{2}) with initial function
$(|u_0(x)|-r)^+$. By Proposition~\ref{pro1} we have the estimate:
$$
\int_{|x|<R} (|u_\nu(t,x)|-r)^+dx\le\int_{|x|<R+Nt}(|u_0(x)|-r)^+dx\mathop{\to}_{r\to+\infty} 0.
$$
By Danford-Pettis criterion, this estimate implies weak compactness of the sequence  $u_\nu(t,x)$ in $L^1_{loc}(\bar\Pi)$. Therefore, there exists a subsequence $u_k=u_{\nu_k}(t,x)$, $k\in\N$, with $\nu_k\to\infty$ as $k\to\infty$ such that $u_k\mathop{\rightharpoonup}\limits_{k\to\infty} u=u(t,x)$ weakly in $L^1_{loc}(\bar\Pi)$.
Since the sequence $a_k(x)\doteq a_{\nu_k}(x)\to a(x)$ as $k\to\infty$ strongly in $L^1_{loc}(\R^n,\R^n)$ and this
sequence is uniformly bounded, then $u_k(t,x)a_k(x)\mathop{\rightharpoonup}\limits_{k\to\infty} u(t,x)a(x)$ weakly in $L^1_{loc}(\bar\Pi,\R^n)$. This allows to pass to the limit as $k\to\infty$ in relation (\ref{3}) corresponding to problem (\ref{13}), (\ref{2}):
$$
\int_\Pi [u_kf_t+u_ka_k\cdot\nabla_x f]dtdx+\int_{\R^n} u_0(x)f(0,x)dx=0 \quad \forall f=f(t,x)\in C_0^1(\bar\Pi)
$$
and obtain that
$$
\int_\Pi [uf_t+ua\cdot\nabla_x f]dtdx+\int_{\R^n} u_0(x)f(0,x)dx=0 \quad \forall f=f(t,x)\in C_0^1(\bar\Pi).
$$
By Definition~\ref{def1}, this means that $u$ is a g.s. of original problem (\ref{1}), (\ref{2}).
We established the existence of a g.s. to (\ref{1}), (\ref{2}) for arbitrary initial function $u_0\in L^1_{loc}(\R^n)$ (~in the case $u_0\in L^\infty(\R^n)$ this follows from \cite[Theorem~1]{PaTr}~). Concerning the uniqueness, generally it fails, see examples in \cite{Br,CLR,PaTr}. It is clear, that the uniqueness follows from
the renormalization property. Indeed, let $u(t,x)\in L^1_{loc}(\bar\Pi)$ be a g.s. of (\ref{1}), (\ref{2}) with zero initial data. Then $|u(t,x)|$ be a nonnegative g.s. of the same problem. By Proposition~\ref{pro1} we see that for a.e. $t>0$
$$
\int_{|x|<R} |u(t,x)|dx\le\int_{|x|<R+Nt}|u_0(x)|dx=0 \quad \forall R>0,
$$
which implies that $u=0$ a.e. on $\Pi$. By the linearity the uniqueness follows.

Suppose that the following requirement is fulfilled.

\begin{itemize}
\item[(R)]
Any g.s. $u(t,x)$ of (\ref{1}), (\ref{2}) such that $u_0,u(t,\cdot)\in L^2$, $\|u(t,\cdot)\|_2\le\const$, satisfies the renormalization property.
\end{itemize}

As we will demonstrate below in this case g.s. of (\ref{1}), (\ref{2}) form the $C_0$-semigroup $T_t=e^{-At}$ governed by a skew-adjoint generator $A=-A^*$. First, we prove that trajectories $T_tu_0$ of such semigroups are necessary g.s. of (\ref{1}), (\ref{2}). More precisely, the following criterion holds.

\begin{lemma}\label{lem1}
Let $B$ be an infinitesimal generator of $C_0$-semigroup $T_t$ in $L^2$. Then
the function $u(t,x)=T_tu_0(x)$ is a g.s. of problem (\ref{1}), (\ref{2}) for every $u_0\in L^2$ if and only if $B\subset A^*$.
\end{lemma}

\begin{proof}
First, we assume that $B\subset A^*$ and $u_0\in D(B)$. Then $u(t,\cdot)=T_tu_0(x)$ is a $C^1$-function with values in $L^2$: $\dot u=BT_tu_0=Bu(t,\cdot)$. This implies that for arbitrary $g=g(x)\in C_0^1(\R^n)$
$$
\frac{d}{dt}(u(t,\cdot),g)_2=(Bu(t,\cdot),g)_2=(A^*u(t,\cdot),g)_2=(u(t,\cdot),Ag)_2,
$$
where $Ag=\div ag=a\cdot\nabla g$, that is,
$$
\frac{d}{dt}\int_{R^n}u(t,x)g(x)dx-\int_{\R^n}u(t,x)a(x)\cdot\nabla_xg(x)dx=0.
$$
Multiplying this relation by a function $h(t)\in C_0^1([0,+\infty))$ and integrating over $t$, we obtain with the help of integration by part formula that
\begin{equation}\label{l1}
\int_{\R^n} u_0(x)f(0,x)dx+\int_{\Pi}u[f_t+a\cdot\nabla_x f]dtdx=0,
\end{equation}
where $f=g(x)h(t)$. Since the linear span of such functions $f$ is dense in $C_0^1(\bar\Pi)$, we see that (\ref{l1}) holds for every $f=f(t,x)\in C_0^1(\bar\Pi)$. Hence, $u(t,x)$ is a g.s. of (\ref{1}), (\ref{2}). If $u_0(x)\in L^2$ is an arbitrary function, then we can find a sequence $u_{0k}\in D(B)$ converging to $u_0$ as $k\to\infty$ in $L^2$
(~notice that by the Hille-Yosida theorem $D(B)$ is dense in $L^2$~). Then $u_k(t,x)=T_tu_{0k}(x)$ are g.s. of (\ref{1}), (\ref{2}) with initial data $u_{0k}$, $k\in\N$, and
$$
\|u_k(t,\cdot)-u(t,\cdot)\|_2\le\|T_t\||u_{0k}-u_0\|_2\mathop{\to}_{k\to\infty} 0
$$
uniformly in $t$ on any segment $[0,T]$. In particular $u_k\to u$ as $k\to\infty$ in $L^1_{loc}(\bar\Pi)$.
Passing to the limit as $k\to\infty$ in the relation
$$
\int_{\R^n} u_{0k}(x)f(0,x)dx+\int_{\Pi}u_k[f_t+a\cdot\nabla_x f]dtdx=0, \quad f=f(t,x)\in C_0^1(\bar\Pi),
$$
we arrive at the identity (\ref{l1}). Therefore, $u(t,x)$ is a g.s. of (\ref{1}), (\ref{2}), as was to be proved.

Conversely, assume that all the functions $u(t,x)=T_tu_0$, $u_0\in L^2$, are g.s. of (\ref{1}), (\ref{2}). If $u_0\in D(B)$, then $u(t,\cdot)=T_tu_0\in C^1([0,+\infty),L^2)$, and $u'(0)=Bu_0$. This implies that for each function $g(x)\in C_0^1(\R^n)$
the scalar function
\begin{equation}\label{l4}
I(t)=\int_{\R^n}u(t,x)g(x)dx=(u(t,\cdot),g)_2\in C^1([0,+\infty)), \ I'(0)=(g,Bu_0)_2.
\end{equation}
On the other hand for all $h(t)\in C_0^1([0,+\infty))$
\begin{eqnarray*}
\int_0^{+\infty}I(t)h'(t)dt=\int_{\Pi} u(t,x)g(x)h'(t)dtdx=\\
-h(0)\int_{\R^n} u_0(x)g(x)dx-\int_{\Pi} u(t,x) a(x)\cdot\nabla g(x) h(t)dxdt,
\end{eqnarray*}
by virtue of (\ref{3}) with $f=h(t)g(x)$. Taking in this relation $h(t)=\theta_\nu(t_0-t)$ and passing to the limit as $\nu\to\infty$ we obtain the equality
$$
I(t_0)-I(0)=\int_0^{t_0}\int_{\R^n}u(t,x) a(x)\cdot\nabla g(x)dxdt=\int_0^{t_0}(Ag,u(t,\cdot))_2dt,
$$
which implies the relation $I'(0)=(Ag,u_0)_2$. In view of (\ref{l4}) we find $(Ag,u_0)_2=(g,Bu_0)_2$ for all $g\in C_0^1(\R^n)$. Therefore, $u_0\in D(A^*)$ and $A^*u_0=Bu_0$. Hence $B\subset A^*$. The proof is complete.
\end{proof}

Now we are ready to prove the following statement analogous to Theorem~\ref{th1} (~that is, the necessity statement in Theorem~\ref{main}~).

\begin{theorem}\label{th2} Suppose that assumption (R) is satisfied.
Then the operator $A$ (~recall that it is the closure of operator $\div au$, $u\in C_0^1(\R^n)$~) is skew-adjoint.
\end{theorem}

\begin{proof}
Let $A_\nu$ be the closure of operator $\div (a_\nu u)$, where $a_\nu(x)=a*\gamma_\nu(x)$, $\nu\in\N$, is the above defined sequence of averaged coefficients. If $u_0(x)\in L^2(\R^n)$ and $u_\nu=u_\nu(t,x)$ is a unique g.s. of the approximate problem (\ref{13}), (\ref{2}), then $(u_\nu)^2$ is a g.s. of (\ref{13}), (\ref{2}) with initial data
$(u_0)^2\in L^1(\R^n)$ in view of Proposition~\ref{pro2}(i). We know that there exists a subsequence (not relabeled) such that
$u_\nu\rightharpoonup u$, $(u_\nu)^2\rightharpoonup v$ as $\nu\to\infty$ weakly in $L^1_{loc}(\bar\Pi)$, where
$u,v$ are g.s. of original problem (\ref{1}), (\ref{2}) with initial data $u_0$, $(u_0)^2$, respectively.
Observe that since a g.s. of problem (\ref{1}), (\ref{2}) is unique, then the above limit relations remain valid for the original sequences, without extraction of subsequences. By the renormalization property we have $v=u^2$, which implies the strong convergence $u_\nu\mathop{\to}\limits_{\nu\to\infty} u$ in $L^2_{loc}(\bar\Pi)$. Indeed,
in view of Proposition~\ref{pro2}(iv) $\int_0^T |u_\nu(t,x)|^2dtdx=T\|u_0\|_2$, therefore the sequence $u_\nu$ is bounded in $L^2_{loc}(\bar\Pi)$. This readily implies that this sequence converges to $u$ weakly in $L^2_{loc}(\bar\Pi)$. Hence, for each nonnegative $\rho(t,x)\in C_0(\bar\Pi)$
$$
\int_\Pi (u_\nu-u)^2\rho dtdx=\int_\Pi ((u_\nu)^2-u^2)\rho dtdx-2\int_\Pi (u_\nu-u)u\rho dtdx\mathop{\to}_{\nu\to\infty} 0.
$$
Thus, $u_\nu\mathop{\to}\limits_{\nu\to\infty} u$ in $L^2_{loc}(\bar\Pi)$. Extracting a subsequence (not relabeled)
we can assume that for almost all $t>0$ $u_\nu(t,\cdot)\to u(t,\cdot)$ as $\nu\to\infty$ in $L^2_{loc}(\R^n)$.
By estimate (\ref{6a}) we can find sufficiently large $R>NT$ such that for a.e. $t\in (0,T)$
\begin{eqnarray*}
\int_{|x|>R}(u_\nu(t,x))^2dx\le\int_{|x|>R-Nt}(u_0(x))^2dx<\varepsilon/4, \\
\int_{|x|>R}(u(t,x))^2dx\le\int_{|x|>R-Nt}(u_0(x))^2dx<\varepsilon/4,
\end{eqnarray*}
where $\varepsilon$ is an arbitrary positive number. This implies that
\begin{eqnarray*}
\int_{\R^n}(u_\nu(t,x)-u(t,x))^2dx\le\int_{|x|<R}(u_\nu(t,x)-u(t,x))^2dx+\\ \int_{|x|>R}(u_\nu(t,x)-u(t,x))^2dx\le
\int_{|x|<R}(u_\nu(t,x)-u(t,x))^2dx+ \\
2\int_{|x|>R}(u_\nu(t,x))^2dx+2\int_{|x|>R}(u(t,x))^2dx\le \\
\int_{|x|<R}(u_\nu(t,x)-u(t,x))^2dx+ \varepsilon.
\end{eqnarray*}
Since $u_\nu(t,\cdot)\mathop{\to}\limits_{\nu\to\infty} u(t,\cdot)$ in $L^2_{loc}(\R^n)$, we obtain the relation
$$
\limsup_{\nu\to\infty}\int_{\R^n}(u_\nu(t,x)-u(t,x))^2dx\le\varepsilon
$$
for all $\varepsilon>0$. Therefore,
$$
\lim_{\nu\to\infty}\int_{\R^n}(u_\nu(t,x)-u(t,x))^2dx=0,
$$
that is, $u_\nu(t,\cdot)\to u(t,\cdot)$ as $\nu\to\infty$ in $L^2$ for a.e. $t>0$.

Let us show that actually this convergence is uniform with respect to $t$ on any segment $[0,T]$.
For that we use estimate (\ref{mod}) with $p=2$. By this estimate for all $\nu\in\N$
\begin{equation}\label{mod2}
\|u_\nu(t+h,\cdot)-u_\nu(t,\cdot)\|_2\le\omega_2(h)=\inf_{v\in C_0^1(\R^n)}\left(2\|u_0-v\|_2+NC(v)|h|\right).
\end{equation}
Since the above estimate is uniform in $\nu$ and $u_\nu(t,\cdot)\mathop{\to}\limits_{\nu\to\infty} u(t,\cdot)$ in $L^2$ for a.e. $t>0$, we conclude that this convergence holds for all $t>0$ and it is unform on any segment $[0,T]$.
From (\ref{mod2}) it follows in the limit as $\nu\to\infty$ that
$$
\|u(t+h,\cdot)-u(t,\cdot)\|_2\le\omega_2(h) \quad \forall t,t+h\ge 0.
$$
Thus, the operators $T_tu_0=u(t,\cdot)$ form a $C_0$-semigroup of linear operators on $L^2$, and the sequence of the unitary groups $T_t^\nu u_0=u_\nu(t,\cdot)$ converges to $T_t$ uniformly on any segment $[0,T]$.
It is clear that $\|T_tu_0\|_2=\lim\limits_{\nu\to\infty}\|T_t^\nu u_0\|_2=\|u_0\|_2$.
Observe that by the same reasons as above we can establish that for each $\tau>0$ the sequence $\tilde u_\nu(t,\cdot)=T^\nu_{t-\tau}u_0$ converges uniformly on $[0,\tau]$ to a g.s. $\tilde u(t,x)$ of problem (\ref{1}), (\ref{2}) with some initial function $\tilde u_0(x)$. By the construction $T_\tau\tilde u_0=\tilde u(\tau,\cdot)=u_0$. We see that the operator $T_\tau$ is invertible, $\tilde u_0=(T_\tau)^{-1}u_0$. Hence $T_t$ are unitary operators and they form the unitary group $T(t)$ (for negative $t$ we set $T(t)=(T(-t))^{-1}=(T(-t))^*$).
By Stone' theorem the infinitesimal generator $B$ of this group is a skew-adjoint operator on $L^2$. By the Trotter--Kato theorem, the convergence $T_t^\nu\to T_t$ of semigroups, which we have established above, implies the convergence of the resolvents $(E+A_\nu)^{-1}u\to (E-B)^{-1}u$ in $L^2$ as $\nu\to\infty$. Recall that $A_\nu$
is the closure of operator $\div (a_\nu u)$, $u\in C_0^1(\R^n)$. By Theorem~\ref{th1} this operator is skew-adjoint and $-A_\nu$ is the generator of semigroup (group) $T_t^\nu$. Denote $v_\nu=(E+A_\nu)^{-1}u$,
$v=(E-B)^{-1}u$. Then $v_\nu\to v$ as $\nu\to\infty$ in $L^2$ and $v_\nu+A_\nu v_\nu=v-Bv=u$. Therefore,
$A_\nu v_\nu\to -Bv$ as $\nu\to\infty$ in $L^2$. Since $A_\nu=-(A_\nu)^*$, we claim that in $\D'(\R^n)$
$A_\nu v_\nu=\div(a_\nu(x)v_\nu(x))\mathop{\to}\limits_{\nu\to\infty} -Bv$. Passing to the limit as $\nu\to\infty$,
we obtain $Bv=-\div(av)$, that is, $v\in D(A^*)$, $Bv=A^*v$. Hence, $B\subset A^*$ and $A=A^{**}\subset B^*=-B$, so that $B$ is a skew-adjoint extension of the skew-symmetric operator $-A$. If $B\not=-A$ then this extension  cannot be unique (because the deficiency indices of the symmetric operator $-iA$ are identical and nonzero).
If $\tilde B$ is another skew-adjoint extension of $-A$ then $\tilde B=-\tilde B^*\subset A^*$. The operator $\tilde B$ generates the unitary group $\tilde T_t=e^{\tilde Bt}$ different of $T_t$ (since $\tilde B\not=B$). Therefore, we can find $u_0\in L^2$ such that $\tilde u(t,x)=\tilde T(t)u_0(x)\not\equiv u(t,x)=T_tu_0(x)$. However, in view of Lemma~\ref{lem1} both functions $\tilde u(t,x)$, $u(t,x)$ are g.s. of the same Cauchy problem (\ref{1}), (\ref{2}). By the uniqueness we see that $\tilde u\equiv u$. The obtained contradiction shows that $B=-A$.
Hence, the operator $A=-B$ is skew-adjoint, as was to be proved.
\end{proof}

\section{The group solutions}\label{sec4}

We are going to establish the inverse statement to Theorem~\ref{th2} claiming that if the operator $A$ is skew-adjoint, then any g.s. $u(t,x)\in L^2_{loc}(\bar\Pi)$ of problem (\ref{1}), (\ref{2}) satisfies the renormalization property.

Observe that in this case $T_t=e^{-At}$ is an unitary $C_0$-group on $L^2$ governed by the skew-adjoint operator $-A$. We call a function $u(t,x)=T_tu_0(x)$
\textit{a group solution} of problem (\ref{1}), (\ref{2}). By Lemma~\ref{lem1} the group solution is a g.s. of (\ref{1}), (\ref{2}).
First, we establish that the approximate sequence $u_\nu=T_t^\nu u_0(x)$ converges strongly as $\nu\to\infty$
to the group solution.

\begin{proposition}\label{pro3}
Let $T_t^\nu =e^{-A_\nu t}$ be the group with generator $-A_\nu$ (~being the closure of operator $-\div (a_\nu u)$~), so that $T_t^\nu u_0=u_\nu(t,x)$
is the unique g.s. of approximate problem (\ref{13}), (\ref{2}). Then $u_\nu(t,\cdot)\to u(t,\cdot)=T_tu_0$ as $\nu\to\infty$ in $L^2$ uniformly on any segment $|t|\le T$.
\end{proposition}

\begin{proof}
Assume that $f\in L^2$, $h\not=0$. We set $v_\nu=(E+hA_\nu)^{-1}f\in D(A_\nu)$, $\nu\in\N$; $v=(E+hA)^{-1}f\in D(A)$. Then $v_\nu+hA_\nu v_\nu=f$, $v+hA v=f$. Since $A_\nu=-(A_\nu)^*$, $A=-A^*$, these equalities mean that
\begin{equation}\label{14}
v_\nu(x)+h\div (a_\nu(x) v_\nu(x))=v(x)+h\div (a(x) v(x))=f(x) \ \mbox{ in } \D'(\R^n).
\end{equation}
Since $A_\nu$, $A$ are skew-symmetric,
\begin{eqnarray}\label{15}
\|v_\nu\|_2^2=(f,v_\nu)_2+h(A_\nu v_\nu,v_\nu)_2=(f,v_\nu)_2, \
\|v\|_2^2=\nonumber\\ (f,v)_2+h(Av,v)_2=(f,v)_2.
\end{eqnarray}
From (\ref{15}) it follows  that $\|v_\nu\|_2\le\|f\|_2$ for all $\nu\in\N$. Therefore, possibly after extraction of a subsequence (not relabeled), we can assume that $v_\nu\rightharpoonup w$ as $\nu\to\infty$ weakly in $L^2$, $w=w(x)\in L^2$. Passing to the limit as $\nu\to\infty$ in (\ref{14}) and taking into account that the sequence $a_\nu(x)\mathop{\to}\limits_{\nu\to\infty} a(x)$ in $L^1_{loc}(\R^n)$ and uniformly bounded, we find
$w(x)+h\div (a(x) w(x))=f(x)$ in $\D'(\R^n)$, which means $w+hAw=f$. Hence $v-w+hA(v-w)=0$ and we conclude that
$w=v$ because the operator $E+hA$ is invertible. Thus, $v_\nu\rightharpoonup v$ as $\nu\to\infty$ weakly in $L^2$.
Then $(f,v_\nu)_2\mathop{\to}\limits_{\nu\to\infty} (f,v)_2$ and from (\ref{15}) it follows that
$\|v_\nu\|_2\mathop{\to}\limits_{\nu\to\infty} \|v\|_2$. It is well-known that this implies the strong convergence
$v_\nu\mathop{\to}\limits_{\nu\to\infty} v$ in $L^2$. Notice that the limit function $v$ does not depend on the choice of weakly convergent subsequence. Therefore, the original sequence converges to the same limit strongly in $L^2$. We have established the strong convergence of resolvents $(E+hA_\nu)^{-1}\to (E+hA)^{-1}$. By the Trotter--Kato theorem the sequence of groups $T_t^\nu$ converges to the group $T_t$  in the sense indicated in the formulation of our theorem. The proof is complete.
\end{proof}

\begin{corollary}\label{cor1}
Let $u_0\in L^2(\R^n)$. Then $u(t,x)=T_tu_0(x)$ is a renormalized solution of (\ref{1}), (\ref{2}).
\end{corollary}

\begin{proof}
Let $g(u)$ be a bounded continuous function, $u_\nu(t,x)=T_t^\nu u_0(x)$, $\nu\in\N$ be g.s. of approximate problem
(\ref{13}), (\ref{2}). By Proposition~\ref{pro2}(i) $u_\nu(t,x)$ is a renormalized solution of (\ref{13}), (\ref{2}). Therefore, $g(u_\nu(t,x))$ is a g.s. of (\ref{13}), (\ref{2}) with initial data $g(u_0(x))$, that is, $\forall f=f(t,x)\in C_0^1(\bar\Pi)$
\begin{equation}\label{16}
\int_{\R^n} g(u_0(x))f(0,x)dx+\int_\Pi g(u_\nu(t,x))[f_t(t,x)+a_\nu(x)\cdot\nabla_x f(t,x)]dtdx=0.
\end{equation}
By Proposition~\ref{pro3} the sequence $g(u_\nu(t,x))\to  g(u(t,x))$ as $\nu\to\infty$ in $L^1_{loc}(\bar\Pi)$, which allows to pass to the limit as $\nu\to\infty$ in (\ref{16}) and obtain the relation: $\forall f=f(t,x)\in C_0^1(\bar\Pi)$
\begin{equation}\label{17}
\int_{\R^n} g(u_0(x))f(0,x)dx+\int_\Pi g(u(t,x))[f_t(t,x)+a(x)\cdot\nabla_x f(t,x)]dtdx=0,
\end{equation}
showing that $g(u)$ is a g.s. of (\ref{1}), (\ref{2}).
Consider now the general case $g(u)\in C(\R)$, $g(u_0(x))\in L^1_{loc}(\R^n)$, $g(u(t,x))\in L^1_{loc}(\bar\Pi)$.
Let $g_k(u)=\max(-k,\min(g(u),k))$, $k\in\N$, be cut-off functions. Then $g_k(u)\in C(\R)$, $|g_k(u)|\le k$, $g_k(u)\mathop{\to}\limits_{k\to\infty} g(u)$ $\forall u\in \R$,
$|g_k(u)|=\min(|g(u)|,k)\le |g(u)|$. The latter implies the estimates $|g_k(u_0(x))|\le |g(u_0(x))|$, $|g_k(u(t,x))|\le |g(u(t,x))|$. As we already proved, $g_k(u(t,x)$ are g.s. of (\ref{1}), (\ref{2}) with initial functions $g_k(u_0(x))$. Therefore, identity (\ref{17}) holds with $g=g_k$. Passing to the limit in this relation as $k\to\infty$, with the help of Lebesgue dominated convergence theorem, we arrive at the same identity (\ref{17}) with the limit function $g$. We conclude that $g(u)$ is a g.s. of (\ref{1}), (\ref{2}) with initial data $g(u_0)$.
Thus, $u$ is a renormalized solution of (\ref{1}), (\ref{2}).
\end{proof}

\begin{corollary}\label{cor2}. Assume that the operator $A$ is skew-adjoint. Then for every $u_0(x)\in L^2_{loc}(\R^n)$ there exists a renormalized solution $u(t,x)\in L^2_{loc}(\bar\Pi)$ of the problem (\ref{1}), (\ref{2}).
\end{corollary}

\begin{proof}
Let $u_r=u_r(t,x)\in C(\R,L^2(\R^n))$ be a group solution of (\ref{1}), (\ref{2}) with initial function $u_{0r}=u_0(x)\theta(r-|x|)\in L^2(\R^n)$ (recall that $\theta(s)$ is the Heaviside function). By Corollary~\ref{cor1} $u_r(t,x)$ is a renormalized solution of (\ref{1}), (\ref{2}) for each $r\in\N$.
Since the difference $u_l-u_r$ is a group solution and, therefore, also a renormalized solution of problem (\ref{1}), (\ref{2}) with initial data $u_{0l}-u_{0r}$, $l,r\in\N$, then $|u_l-u_r|$ is a nonnegative g.s. of this problem with initial function $|u_{0l}-u_{0r}|$. By Proposition~\ref{pro1}, we find that for all $t>0$
$$
\int_{|x|<r-Nt} |u_l(t,x)-u_r(t,x)|dx\le\int_{|x|<r} |u_{0l}(x)-u_{0r}(x)|dx=0, \ \forall l>r,
$$
and $u_l(t,x)=u_r(t,x)$ almost everywhere in the cone $C_r=\{ \ (t,x)\in\Pi \ | \ |x|<r-Nt \ \}$.
This implies that the sequence $u_r$ converges as $r\to\infty$ to a function $u=u(t,x)$, where $u=u_r(t,x)$ whenever $(t,x)\in C_r$ for some $r\in\N$. It is clear that $u(t,x)\in L^2_{loc}(\bar\Pi)$. Let us demonstrate that $u$ is the desired renormalized solution. Let a function $g(u)\in C(\R)$ be such that $g(u_0(x))\in L^1_{loc}(\R^n)$, $g(u(t,x))\in L^1_{loc}(\bar\Pi)$, and $f=f(t,x)\in C_0^1(\bar\Pi)$. Then one can choose a sufficiently large $r\in\N$ such that
$\supp f\subset C_r$. Since $u=u_r$ in $C_r$ while $u_r$ is a renormalized solution, we conclude that
\begin{eqnarray*}
\int_{\R^n} g(u_0(x))f(0,x)dx+\int_\Pi g(u(t,x))[f_t+a(x)\cdot\nabla_x f]dtdx=\\
\int_{\R^n} g(u_{0r}(x))f(0,x)dx+\int_\Pi g(u_r(t,x))[f_t+a(x)\cdot\nabla_x f]dtdx=0.
\end{eqnarray*}
Hence, $u$ is a renormalized solution of (\ref{1}), (\ref{2}).
\end{proof}

\begin{theorem}\label{th3} Assume that $A$ is a skew-adjoint operator, and $\div (a(x)u(x))=0$ in $\D'(\R^n)$,
where $u(x)\in L^2_{loc}(\R^n)$.
Then $\div (a(x)g(u(x)))=0$ in $\D'(\R^n)$ for any $g(u)\in C(\R)$ such that $g(u(x))\in L^1_{loc}(\R^n)$.
\end{theorem}

\begin{proof}
Let $p(y)\in C_0^1(\R^n)$ be a function equaled $1$ in the unit ball
$|y|^2\le 1$. We set $u_r(x)=u(x)p(x/r)\in L^2(\R^n)$. By our assumption the operator $A$ is skew-adjoint and, in view of equality $A=-(A)^*$, this operator may be considered in distributional sense. Obviously, for all $r>0$
\begin{eqnarray*}
Au_r(s,x)=v_r(x)\doteq u(x)A p(x/r)=\frac{1}{r}u(x)a(x)\cdot(\nabla_y p)(x/r) \ \mbox{ in } \D'(\R^n).
\end{eqnarray*}
Since $v_r(x)\in L^2(\R^n)$, then $u_r(x)\in D(A)$. Now let $U_r(t,x)=e^{-At}u_r(x)$ be the group solution of (\ref{1}), (\ref{2}) with initial data $u_r(x)$. As we demonstrated above, $u_r(x)\in D(A)$. Therefore, $U_r(t,\cdot)\in C^1(\R,L^2(\R^n))$, and
$$
V_r\doteq\frac{d}{dt}U_r(t,\cdot)=-e^{-At}Au_r=-e^{-At}v_r.
$$
We see that  $V_r(t,x)$ is a renormalized solution to the Cauchy problem (\ref{1}), (\ref{2}) with initial data $-v_r(x)$. By Corollary~\ref{cor1} $|V_r(t,x)|$ is a g.s. of this problem with initial function $|v_r(x)|$.
Let $T,R>0$, $r>R+NT$. Then by Proposition~\ref{pro1} for all $t\in [0,T]$
$$
\int_{|x|<R}|V_r(t,x)|dx\le \int_{|x|<r}|v_r(x)|dx=0
$$
(since $(\nabla_y p)(x/r)=0$ for $|x|<r$).

We find that $V_r=\frac{d}{dt}U_r\equiv 0$ in the cylinder $C_{R,T}=\{ \ (t,x) \ | \ |x|<R, \ t\in (0,T) \ \}$.
This implies that $U_r\equiv u_r=u$ in this cylinder. Now, let $g(u)$ be a bounded continuous function. By Corollary~\ref{cor1} the function $g(U_r)$ is a g.s. of (\ref{1}), (\ref{2}). Therefore this function satisfies (\ref{1}) in $\D'(C_{R,T})$. Since $g(U_r)\equiv g(u)$ in $C_{R,T}$, we obtain that $\div_x (ag(u))=0$ in $\D'(V_R)$,
where $V_R$ denotes the open ball $|x|<R$. In view of arbitrariness of $R$ we conclude that $\div_x (ag(u))=0$ in $\D'(\R^n)$. In the general case when $g(u)\in C(\R)$, $g(u(t,x))\in L^1_{loc}(\R^n)$, we construct the sequence
of cut-off functions $g_k(u)=\max(-k,\min(g(u),k))$. Then $\div_x (ag_k(u))=0$ in $\D'(\R^n)$ for all $k\in\N$.
Since $g_k(u(x))\to g(u(x))$ as $k\to\infty$ in $L^1_{loc}(\R^n)$ (cf. the proof of Corollary~\ref{cor1}), we
can pass to the limit as $k\to\infty$ in the relation $\div_x (ag_k(u))=0$ and conclude that $\div_x (ag(u))=0$  in $\D'(\R^n)$.
\end{proof}

\section{Main result: the sufficiency}\label{sec5}
We are going to establish the much stronger result than the statement of Corollary~\ref{cor1}, claiming that any generalized solution $u(t,x)$ of (\ref{1}), (\ref{2}) is a renormalized solution, that is, the sufficiency statement of our main Theorem~\ref{main}.

We define the operator $\tilde A_0=\frac{\partial}{\partial s}+A_0$ acting on $C_0^1(\R^{n+1})$, so that $\tilde A_0u(s,x)=\frac{\partial u(s,x)}{\partial s}+a(x)\nabla_x u(s,x)$. Let $\tilde A$ be a closure of $\tilde A_0$ in $L^2(\R^{n+1})$. We will prove that $\tilde A$ is a skew-adjoint operator whenever $A$ is a skew-adjoint operator on $L^2(\R^n)$. First, we observe that, at least formally, operator $-\tilde A$ should coincide with the infinitesimal generator of the unitary group $G_t u(s,\cdot)=T_t u(s-t,\cdot)$, $u(s,x)\in L^2(\R^{n+1})$.
By Stone's theorem $G_t=e^{-Bt}$, where $B$ is a skew-adjoint operator on $L^2(\R^{n+1})$. The following statement justifies this formal observation.

\begin{lemma}\label{lem2}
The equality $\tilde A=B$ holds. In particular, the operator $\tilde A$ is skew-adjoint.
\end{lemma}

\begin{proof}
We denote by $X$ the space $L^2(\R^n)$ and by $X_0$ the space $D(A)$ equipped with the graph norm $\|x\|_2+\|Ax\|_2$. Since the operator $A$ is closed, $X_0$ is a Banach space. Let $F$ be a subspace of $L^2(\R^{n+1})=L^2(\R,X)$ consisting of
functions $u(s,\cdot)\in L^2(\R,X_0)$, such that $\frac{d}{ds}u(s,\cdot)\in L^2(\R,X)$.
We show that $F\subset D(B)\cap D(\tilde A)$ and $Bu=\tilde Au$ on $F$. Thus, assume that $u(s,x)\in F$.
Then,
$$
\frac{G_tu-u}{t}=T_t \frac{u(s-t,\cdot)-u(s,\cdot)}{t}+\frac{T_tu(s,\cdot)-u(s,\cdot)}{t}
$$
Since
\begin{eqnarray*}
\lim_{t\to 0} \frac{u(s-t,\cdot)-u(s,\cdot)}{t}=-\frac{d}{ds}u(s,\cdot), \\ \\
\lim_{t\to 0} \frac{T_tu(s,\cdot)-u(s,\cdot)}{t}=-Au(s,\cdot) \ \mbox{ in } L^2(\R,X),
\end{eqnarray*}
we find that there exists
$$
-Bu=\lim_{t\to 0}\frac{G_tu-u}{t}=-\frac{d}{ds}u(s,\cdot)-Au(s,\cdot) \ \mbox{ in } L^2(\R,X),
$$
that is, $u\in D(-B)=D(B)$ and
\begin{equation}\label{18}
Bu=\frac{d}{ds}u(s,\cdot)+Au(s,\cdot).
\end{equation}
Let us show that the same holds for the operator $\tilde A$. Assume firstly that
\begin{equation}\label{19}
u(s,x)=\sum_{j=1}^N \alpha_j(s)v_j(x), \quad \alpha_j(s)\in C_0^1(\R),  \ v_j\in X_0,  \ j=1,\ldots,N.
\end{equation}
If $v_j(x)\in C_0^1(\R^n)$, then $u(s,x)\in C_0^1(\R^{n+1})=D(\tilde A_0)$ and then
$$
\tilde Au(s,x)=\tilde A_0 u(s,x)=\frac{\partial}{\partial s}u(s,x)+a(x)\cdot\nabla_x u(s,x)=\frac{d}{ds}u(s,\cdot)+Au(s,\cdot).
$$
In the case of arbitrary $v_j\in X_0$ we can find sequences $v_{jr}\in C_0^1(\R^n)$, $r\in\N$, converging to $v_j$ as $r\to\infty$ in $X_0$ (because $A$ is the closure of $A_0$). Then the sequences
\begin{eqnarray*}
u_r(s,x)=\sum_{j=1}^N \alpha_j(s)v_{jr}(x)\mathop{\to}_{r\to\infty} u(s,x), \
\tilde Au_r=\sum_{j=1}^N \alpha_j'(s)v_{jr}(x)+ \\ \sum_{j=1}^N \alpha_j(s)Av_{jr}(x)\mathop{\to}_{r\to\infty}
\sum_{j=1}^N \alpha_j'(s)v_j(x)+\sum_{j=1}^N \alpha_j(s)Av_j(x)=\frac{d}{ds}u(s,\cdot)+Au(s,\cdot)
\end{eqnarray*}
in $L^2(\R,X)$. Since the operator $\tilde A$ is closed, we conclude that $u(s,x)\in D(\tilde A)$ and
$\displaystyle\tilde Au(s,\cdot)=\frac{d}{ds}u(s,\cdot)+Au(s,\cdot)$.
Now we consider the general case $u(s,x)\in F$. Then, as is easy to verify, there exists a sequence $u_m(s,x)$, $m\in\N$, of functions having form (\ref{19}) such that $\displaystyle u_m(s,\cdot)\mathop{\to}_{m\to\infty} u(s,\cdot)$ in $L^2(\R,X_0)$, $\displaystyle\frac{d}{ds}u_m(s,\cdot)\mathop{\to}_{m\to\infty}\frac{d}{ds}u(s,\cdot)$ in $L^2(\R,X)$.
Then $\displaystyle u_m(s,\cdot)\mathop{\to}_{m\to\infty} u(s,\cdot)$, $\displaystyle\tilde Au_m(s,\cdot)\mathop{\to}_{m\to\infty}\frac{d}{ds}u(s,\cdot)+Au(s,\cdot)$ in $L^2(\R,X)$, which implies that $u\in D(\tilde A)$, $\tilde Au(s,\cdot)=\frac{d}{ds}u(s,\cdot)+Au(s,\cdot)$ again due to the closedness of $\tilde A$.

In view of (\ref{18}) we conclude that $F\subset D(B)\cap D(\tilde A)$, and $B=\tilde A$ on $F$.
By the known representation of the resolvent $(E+B)^{-1}$, we find
\begin{eqnarray*}
u(s,\cdot)=(E+B)^{-1}f(s,\cdot)=\int_0^{+\infty}e^{-t}G_t fdt=\\ \int_0^{+\infty}e^{-t}T_t f(s-t,\cdot)dt=
\int_{-\infty}^s e^{t-s}T_{s-t} f(t,\cdot)dt.
\end{eqnarray*}
Notice that $X_0$ is an invariant space for a group $T_t$ and since $\|T_tu\|_2=\|u\|_2$, $\|AT_tu\|_2=\|T_tAu\|_2=\|Au\|_2$, then $\|T_t u\|_{X_0}=\|u\|_{X_0}$. Therefore, taking $f(s,x)\in L^2(\R,X_0)$, we find
\begin{eqnarray*}
U(s)\doteq\|u(s,\cdot)\|_{X_0}\le \int_{-\infty}^s e^{t-s}\|T_{s-t} f(t,\cdot)\|_{X_0}dt= \\ \int_{-\infty}^s e^{t-s}\|f(t,\cdot)\|_{X_0}dt=(\gamma*F)(s),
\end{eqnarray*}
where $F(t)=\|f(t,\cdot)\|_{X_0}$, $\gamma(t)=\theta(t)e^{-t}$ (recall that $\theta(t)$ is the Heaviside function). It is clear that $\|\gamma\|_1=1$ and by the known property of convolutions $\|U\|_2\le\|F\|_2$, that is,
$u(s,\cdot)\in L^2(\R,X_0)$, $\|u(s,\cdot)\|_{L^2(\R,X_0)}\le \|f(t,\cdot)\|_{L^2(\R,X_0)}$.
Further, there exists the derivative
\begin{eqnarray*}
\frac{d}{ds}u(s,\cdot)=\frac{d}{ds}\int_{-\infty}^s e^{t-s}T_{s-t} f(t,\cdot)dt=f(s,\cdot)-\int_{-\infty}^s e^{t-s}T_{s-t} f(t,\cdot)dt-\\ \int_{-\infty}^s e^{t-s}AT_{s-t} f(t,\cdot)dt=
f(s,\cdot)-u(s,\cdot)-Au(s,\cdot)\in L^2(\R,X).
\end{eqnarray*}
We see that $u(s,\cdot)\in F$. Assume that $u(s,\cdot)\in D(B)$. Then, there exists a unique $f(s,\cdot)\in L^2(\R,X)$ such that $u(s,\cdot)=(E+B)^{-1}f(s,\cdot)$. Evidently, $L^2(\R,X_0)$ is dense in $L^2(\R,X)$, which implies existence of a sequence $f_k(s,\cdot)\in L^2(\R,X_0)$, $k\in\N$, such that $f_k\to f$ as $k\to\infty$ in $L^2(\R,X)$. We define the corresponding sequence $u_k=u_k(s,\cdot)=(E+B)^{-1}f_k$. Then $u_k\to u$, $\tilde Au_k=Bu_k\to Bu$ as $k\to\infty$ in $L^2(\R,X)$. Since $\tilde A$ is a closed operator, we derive that $u\in D(\tilde A)$ and $\tilde Au=Bu$. Hence, $B\subset\tilde A$. Conversely, $\tilde A_0\subset B$ (since, evidently, $D(\tilde A_0)\subset F$~), which implies $\tilde A\subset B$ as the closure of $\tilde A_0$. We conclude that $\tilde A=B$, as required.
\end{proof}

Now, we are ready to prove the renormalization property.

\begin{theorem}\label{th4}
Assume that operator $A$ is skew-adjoint and $u_0\in L^2_{loc}(\R^n)$. Then any g.s. $u(t,x)\in L^2_{loc}(\bar\Pi)$ of the problem (\ref{1}), (\ref{2}) is a renormalized solution of this problem and, therefore, is unique.
\end{theorem}

\begin{proof}
We may extend $u(t,x)$ to a g.s. of (\ref{1}), (\ref{2}) on the whole space $\R^{n+1}$, setting
$u(-t,x)=v(t,x)$, where $v(t,x)\in L^2_{loc}(\bar\Pi)$ is a renormalized solution of the problem $v_t-\div(a(x)v)=0$, $v(0,x)=u_0(x)$. Since the operator $-A$ is skew-adjoint, this renormalized solution exists due to Corollary~\ref{cor2}. Then $u_t+\div (a(x)u)=0$ in $D'(\R^{n+1})$. By Lemma~\ref{lem2} the operator
$\frac{\partial}{\partial t}+\div (au)$ is skew-adjoint on $L^2(\R^{n+1})$. Then, by Theorem~\ref{th3}
$g(u)_t+\div (a(x)g(u))=0$ in $D'(\R^{n+1})$ whenever $g(u)\in L^1_{loc}(\R^{n+1})$. This easily implies that $u(t,x)$ is a renormalized solution of (\ref{1}), (\ref{2}).
\end{proof}

\begin{remark}\label{rem1}
In the case of more general transport equation
\begin{equation}\label{r1}
u_t+a(t,x)\cdot\nabla_x u=u_t+\div_x (a(t,x)u)=0
\end{equation}
with $a(t,x)=(a_1(t,x),\ldots,a_n(t,x))\in L^\infty(\Pi,\R^n)$, $\div_x a(t,x)=0$,
we may extend the field $a(t,x)$ on the whole space $(t,x)\in\R^{n+1}$, setting $a(t,x)=-a(-t,x)$ for $t<0$.
It is clear that the vector field $\tilde a(t,x)=\frac{\partial}{\partial t}+a(t,x)$ is bounded and solenoidal on $\R^{n+1}$, and for any g.s. $u(t,x)\in L^1_{loc}(\bar\Pi)$ of (\ref{r1}) the function $\tilde u(t,x)=u(|t|,x)$ is a g.s. of (\ref{r1}) in the whole space $\R^{n+1}$.

For equation (\ref{r1}) the following analogue of Theorem~\ref{th1} holds.

\begin{theorem}
Any g.s. of the Cauchy problem (\ref{r1}), (\ref{2}) is a renormalized solution if and only if the operator
$A_0u=\tilde a(t,x)\cdot\nabla u=\frac{\partial}{\partial t} u+a(t,x)\cdot\nabla_x u$, $u=u(t,x)\in C_0^1(\R^{n+1})$, is essentially skew-adjoint.
\end{theorem}

\begin{proof}
Let us consider the extended transport equation
\begin{equation}\label{r2}
v_t+\tilde a(s,x)\cdot\nabla_{s,x} v=v_t+v_s+a(s,x)\cdot\nabla_x v=0,
\end{equation}
where $v=v(t,s,x)$, $t>0$, $(s,x)\in\R^{n+1}$. After the change $u(t,s,x)=v(t+s,t,x)$ we obtain the equation
$$
u_t+a(t,x)\cdot\nabla_x u=0,
$$
which coincides with (\ref{r1}). Therefore, any g.s. of (\ref{r1}) (~which necessarily admits some initial data (\ref{2})~) satisfies the renormalization property if and only if this is true for g.s. of equation (\ref{r2}). By Theorem~\ref{r1}, the latter is equivalent to the essential skew-adjointness of the operator $\tilde a(t,x)\cdot\nabla u$. The proof is complete.
\end{proof}
\end{remark}

 \section{Contraction semigroup, which provides g.s. and a criterion of the uniqueness}

In this section we study the general case when the skew-symmetric operator $A$ is not necessarily skew-adjoint. We proof that in this case there always exists a linear $C_0$-semigroup $T_t$ such that $u(t,x)=T_tu_0$ is a g.s. of (\ref{1}), (\ref{2}), and $\|T_tu_0\|_2\le\|u_0\|_2$ for all $u_0\in L^2$ (i.e., $T_t$ are contractions in $L^2$).
Let $\tilde A$ be a maximal skew-symmetric extension of $A$. Then $A\subset\tilde A\subset-\tilde A^*\subset -A^*$.
Denote by $d_+=d_+(\tilde A)=\codim\Im (E+\tilde A)$, $d_-=d_-(\tilde A)=\codim\Im (E-\tilde A)$ the deficiency indexes of $\tilde A$ (generally, these are cardinal numbers). Since $\tilde A$ is a maximal skew-symmetric operator, either $d_+=0$ or $d_-=0$.
Let us define $B=-\tilde A$ if $d_+=0$, $B=\tilde A^*$ if $d_-=0$ (observe that in the case $d_+=d_-=0$ the operator $\tilde A$ is skew-adjoint and $-\tilde A=\tilde A^*$~).

\begin{theorem}\label{th5}
The operator $B$ generates the semigroup of contractions $T_tu=e^{Bt}$ on $L^2$ such that $u(t,x)=T_tu_0$ is a g.s. of (\ref{1}), (\ref{2}) for every initial data $u_0\in L^2$. Moreover, in the case $d_+=0$ the operators $T_t$ are
isometric, that is $\|T_t u\|_2=\|u\|_2$ $\forall u\in L^2$.
\end{theorem}

\begin{proof}
If $d_+=0$ then $\Im (E+\tilde A)=L^2$ and the operator $B=-\tilde A$ is $m$-dissipative. By the Lumer-–Phillips theorem it generates the semigroup of contractions on $L^2$. Moreover, in this case $B$ is skew-symmetric and the operators $T_t=e^{Bt}$ are isometric.
In the remaining case when $d_-=0$ the operator $\tilde A$ is $m$-dissipative. Then (see \cite{Clem}) the operator $B=A^*$ is also $m$-dissipative and generates the semigroup of contractions. Since $-\tilde A\subset \tilde A^*\subset A^*$,
then $B\subset A^*$ and by virtue of Lemma~\ref{lem1} we conclude that the functions $u(t,x)=T_tu_0(x)$ are g.s. of
(\ref{1}), (\ref{2}).
\end{proof}

The following statement gives the criterion of uniqueness of a contraction semigroups constructed in Theorem~\ref{th5}.

\begin{theorem}\label{th6}
A contraction semigroups $T_t$, which provides g.s. $T_tu_0$, is unique if and only if $A$ is a maximal skew-symmetric operator.
\end{theorem}

\begin{proof}
If the skew-symmetric operator $A$ is not maximal (that is, $d_+(A),d_-(A)>0$), then there exist different maximal skew-symmetric extensions $\tilde A_1,\tilde A_2$, such that $d_+(\tilde A_1)=d_+(\tilde A_2)$, $d_-(\tilde A_1)=d_-(\tilde A_2)$. Then $m$-dissipative operators $B_1$, $B_2$ corresponding to $\tilde A_1$, $\tilde A_2$ are different. By the Hille-Yosida theorem they generates different semigroups. Therefore, the uniqueness assumption implies that $A$ is a maximal skew-symmetric operator. Conversely, suppose that the operator $A$ is maximal and $T_t$ is a contraction semigroup in $L^2$, which provides g.s. of problem (\ref{1}), (\ref{2}). Then, by the Lumer-–Phillips theorem, the infinitesimal generator $C$ of this semigroup is $m$-dissipative (maximal dissipative) and by Lemma~\ref{lem1} $C\subset A^*$.
Since also $-A\subset A^*$, we see that $Cx=-Ax$ $\forall x\in D(C)\cap D(A)$. This allows to define the linear operator $\tilde C$ on $D(\tilde C)=D(C)+D(A)$, setting $\tilde Cw=Cu-Av$ if $w=u+v$, $u\in D(C)$, $v\in D(A)$.
If $w=u_1+v_1=u_2+v_2$, where $u_1,u_2\in D(C)$, $v_1,v_2\in D(A)$, then $u_1-u_2=v_2-v_1\in D(C)\cap D(A)$ and
$C(u_1-u_2)=-A(v_2-v_1)$, which implies the equality $Cu_1-Av_1=Cu_2-Av_2$, showing that the value $\tilde Cw$ does not depend on a representation $w=u+v$, $u\in D(C)$, $v\in D(A)$. Thus, the operator $\tilde C$ is well-defined and by the construction $C\subset\tilde C$, $-A\subset\tilde C$. If $w=u+v$, where $u\in D(C)$, $v\in D(A)$, then
\begin{eqnarray}\label{20}
(\tilde Cw,w)_2=(Cu-Av,u+v)_2=(Cu,u)_2-(Av,v)_2+(v,Cu)_2-(Av,u)_2\nonumber\\ =(Cu,u)_2-(Av,v)_2+(v,A^*u)_2-(Av,u)_2=(Cu,u)_2,
\end{eqnarray}
where we use that $C\subset A^*$ and the relations $(Av,u)_2=(v,A^*u)_2$, $(Av,v)=0$ (we recall that $A$ is skew-symmetric). Since the operator $C$ is dissipative, then $(Cu,u)_2\le 0$ (see \cite{Clem}~) and it follows from (\ref{20}) that $(\tilde Cw,w)_2\le 0$ for all $w\in D(\tilde C)$. This means that $\tilde C$ is a dissipative operator. But $C\subset\tilde C$ while $C$ is a maximal dissipative operator. Therefore, $C=\tilde C$ and in particular $D(\tilde C)=D(C)+D(A)=D(C)$. Hence, $D(A)\subset D(C)$ and $-A\subset C\subset A^*$. We recall that $A$ is a maximal skew-symmetric operator, so that either $d_+(A)=0$ or $d_-(A)=0$. In the first case $\Im (E+A)=L^2$, that is, $-A$ is $m$-dissipative operator. From the relation $-A\subset C$ it now follows that $C=-A=B$.
In the second case $\Im(E-A)=L^2$ and $A$ is an $m$-dissipative operator. By the known property (see \cite{Clem})  $A^*$ is an $m$-dissipative operator as well. Since operator $C$ is also $m$-dissipative, it follows from the relation $C\subset A^*$ that $C=A^*=B$. In both cases $C$ coincides with the operator $B$ from Theorem~\ref{th5}.
This, in turn, implies the uniqueness of the semigroup $T_t$.
\end{proof}

Now we are ready to prove part (ii) of main Theorem~\ref{main} claiming that the uniqueness of any g.s. holds if and only if the operator $A$ is skew-adjoint that, in turn, is equivalent to the renormalization property.
It is clear that the renormalization property for every g.s. implies the uniqueness. The inverse statement is a consequence of the following theorem.

\begin{theorem}\label{th7}
Assume that any g.s. of problem (\ref{1}), (\ref{2}) with $u_0\in L^2$ is unique in the class of g.s. with bounded  $\|u(t,\cdot)\|_2$. Then these g.s. satisfy the renormalization property and, therefore, the operator $A$ is skew-adjoint.
\end{theorem}

\begin{proof}
It is clear that the uniqueness assumption implies the uniqueness of a contraction semigroups $T_t$, which provides g.s. By Theorem~\ref{th6} the operator $A$ is maximal skew-symmetric, that is, one of its deficiency indexes $d_+$ or $d_-$ is zero. In view of Theorem~\ref{th5} in the case $d_+=0$ the semigroup $T_t$ consists of isometric embeddings. Therefore, the g.s. $u=u(t,x)=T_tu_0(x)$ satisfies the property: $\|u(t,\cdot)\|_2=\|u_0\|_2$. Let $\tilde u=\tilde u(t,x)$ be a weak limit of a subsequence of g.s. $u_k(t,x)$ to the approximate problem (\ref{13}), (\ref{2}). Since $\|u_k(t,\cdot)\|_2=\|u_0\|_2$, then
\begin{equation}\label{21}
\|u_k\|_{L^2(\Pi_T)}=\sqrt{T}\|u_0\|_2=\|u\|_{L^2(\Pi_T)},
\end{equation}
where $\Pi_T=(0,T)\times\R^n$.
Since $u$, $\tilde u$ are g.s. of the same problem (\ref{1}), (\ref{2}), then by the uniqueness assumption
$u=\tilde u$. Hence $u_k\rightharpoonup u$ as $k\to\infty$ weakly in $L^2(\Pi_T)$ while in view of (\ref{21})
$\|u\|_{L^2(\Pi_T)}=\|u_k\|_{L^2(\Pi_T)}$ for all $k\in\N$. By the known property of weak convergence we conclude
that $u_k\to u$ as $k\to\infty$ strongly in $L^2(\Pi_T)$ for all $T>0$. As in the proof of Corollary~\ref{cor1},
this implies that $u$ is a renormalized solution of (\ref{1}), (\ref{2}). Thus, requirement (R) is fulfilled and by Theorem~\ref{th2} the operator $A$ is skew-adjoint.

Now we consider the case when $d_-=0$. In this case the operator $-A$ generates the semigroup $S_t$ of isometries in $L^2$. We choose $T>0$ and set
$$
u=u(t,x)=\left\{\begin{array}{lr} v(T-t,x), & 0\le t<T, \\ \bar u(t-T,x), & t\ge T, \end{array}\right.
$$
where $v(t,x)=S_tv_0(x)$ and $\bar u=\bar u(t,x)$ is a g.s. of (\ref{1}), (\ref{2}) with initial data $v_0\in L^2$. It is easy to verify that $u(t,x)$ is a g.s. of problem (\ref{1}), (\ref{2}) with the initial function $u_0=\tilde u(T,\cdot)$. By the uniqueness of this g.s. $u=\tilde u$, where, as above, $\tilde u=\tilde u(t,x)$ is a weak limit of the sequence $u_k(t,x)$ of g.s. to approximate problem (\ref{13}), (\ref{2}). We see that
$$
\|\tilde u\|_{L^2(\Pi_T)}=\|u\|_{L^2(\Pi_T)}=\|v\|_{L^2(\Pi_T)}=\sqrt{T}\|u_0\|_2=\|u_k\|_{L^2(\Pi_T)} \ \forall k\in\N.
$$
As was shown in the first part of our proof, this implies the strong convergence $u_k\mathop{\to}\limits_{k\to\infty} u$ in $L^2(\Pi_T)$ and, therefore, the renormalization property. By the latter we find that $v(t,x)$ is a renormalized solution of the Cauchy problem for the equation $v_t-\div av=0$ with initial data $v_0$ (we also take into account that $T>0$ is arbitrary). Thus, requirement (R) for this equation is satisfied and by Theorem~\ref{th2}
we conclude that the operator $-A$ is skew-adjoint. This, in turn, implies that $A$ is a skew-adjoint operator.
By Theorem~\ref{th4} we see that any g.s. of (\ref{1}), (\ref{2}) is a renormalized solution of this problem as well. The proof is complete.
\end{proof}

\section{Generalized characteristics}
We assume that the operator $A$ is skew-adjoint. By Theorem~\ref{th4} for every $u_0(x)\in L^\infty=L^\infty(\R^n)$ there exists a unique g.s. $u(t,x)\in L^\infty(\Pi)$ of the problem (\ref{1}), (\ref{2}), and this g.s. is a renormalized solution as well. It is clear that $\|u\|_\infty\le M\doteq\|u_0\|_\infty$ (this can be derived from the renormalization property. Indeed, $v=(|u|-M)^+$ ia a g.s. of   (\ref{1}), (\ref{2}) with initial data $(|u_0|-M)^+=0$, which implies that $v=0$, i.e., $|u|\le M$~). As readily follows from the definition of g.s. and the renormalization property, the functions $t\to p(u(t,\cdot))$ are weakly continuous on some set of full measure for every $p(u)\in C(\R)$, which implies that the map $t\to u(t,\cdot)$ is strongly continuous in $L^1_{loc}(\R^n)$. In particular, after possible correction of $u$ on the set of null measure, we may and will assume that the functions $u(t,\cdot)\in L^\infty$ are well-defined for all $t\ge 0$ and depend continuously on $t$ (in the space $L^1_{loc}(\R^n)$~). Let $u_1=u_1(t,x)$, $u_2=u_2(t,x)$ be g.s. of problem (\ref{1}), (\ref{2})  with initial functions $u_{01}=u_{01}(x)$, $u_{02}=u_{02}(x)$, respectively. Then, by the renormalization property $u_1u_2=[(u_1+u_2)^2-u_1^2-u_2^2]/2$ is a g.s. of (\ref{1}), (\ref{2}) with the initial data $u_{01}u_{02}=[(u_{01}+u_{0}2)^2-u_{01}^2-u_{02}^2]/2$. Hence, the map $T_t(u_0)=u(t,\cdot)$ is a homomorphism of the algebra $L^\infty$: $T_t(uv)=T_tuT_tv$ for all $u,v\in L^\infty(\R^n)$. Obviously, the semigroup $T_t$ can be extended to the group $T_t$ of isomorphisms of $L^\infty$. These isomorphisms generate the corresponding homeomorphisms $y_t:\S\to\S$ of the spectrum $\S$ of $C^*$-algebra $L^\infty$, so that
\begin{equation}\label{22}
\widehat{u(t,\cdot)}(X)=\widehat{u_0}(y_t(X)) \quad \mbox{ for all } X\in\S,
\end{equation}
 where $\widehat{u}\in C(\S)$ denotes the Gelfand transform of $u\in L^\infty$: $\widehat u(X)=\<X,u\>$
(recall that $\S$ consists on multiplicative functionals $X:L^\infty\to\C$~). Denote by $x_t:\S\to\S$ the inverse homeomorphism $x_t=y_t^{-1}$. Then (\ref{22}) can be written as
$$\widehat{u(t,\cdot)}(x_t(X_0))=\widehat{u_0}(X_0) \quad \forall X_0\in\S,$$ that is, $\widehat{u(t,\cdot)}$ remains constant on the curve $X(t)=x_t(X_0)$, $t\in\R$. It is natural to call this curve the generalized characteristic of equation (\ref{1}). In other words, $X(t)$ can be considered as a generalized solution to characteristic system (\ref{10}) (extended to $\S$) with initial data $X(0)=X_0$.

Let us describe the spectrum $\S$. The below characterization of $\S$ is rather well-known but we cannot find the appropriate references and, therefore, give the description of $\S$
in details. First of all, we introduce the notion of essential ultrafilter.

We call sets $A,B\subset\R^n$ equivalent: $A\sim B$ if $\mu(A\vartriangle B)=0$, where $A\vartriangle B=(A\setminus B)\cup (B\setminus A)$ is the symmetric difference and $\mu$ is the outer Lebesgue measure.
Let $\mathfrak{F}$ be a filter in $\R^n$. This filter is called \textit{essential} if from the conditions
$A\in\mathfrak{F}$ and $B\sim A$ it follows that $B\in\mathfrak{F}$. It is clear that an essential filter cannot include sets
of null measure, since such sets are equivalent to $\emptyset$.
Using Zorn's lemma, one can prove that any essential filter is contained in a maximal essential filter. Maximal
essential filters are called essential ultrafilters.

\begin{lemma}\label{lem3} Let $\mathfrak{U}$ be an essential ultrafilter. Then for each $A\subset\R^n$ either $A\in\mathfrak{U}$ or $\R^n\setminus A\in\mathfrak{U}$.
\end{lemma}

\begin{proof} Assuming that $A\notin\mathfrak{U}$, we introduce
$$\mathfrak{F}=\{ \ B\subset\R^n \ | \ B\cup A\in\mathfrak{U} \ \}.$$
Obviously, $\mathfrak{F}$ is an essential filter, $\R^n\setminus A\in\mathfrak{F}$, and $\mathfrak{U}\le\mathfrak{F}$. Since the filter
$\mathfrak{U}$ is maximal, we obtain that $\mathfrak{U}=\mathfrak{F}$. Hence, $\R^n\setminus A\in\mathfrak{U}$. The proof is complete.
\end{proof}

The property indicated in Lemma~\ref{lem3} is the characteristic property of ultrafilters, see for example, \cite{Burb}. Therefore, we obtain the following statement.

\begin{corollary}\label{cor3} Any essential ultrafilter is an ultrafilter, i.e. a maximal element in a set of all filters. \end{corollary}

\begin{lemma}\label{lem4} Let $\mathfrak{U}$ be an essential ultrafilter, and $f(x)$ be a bounded function in $\R^n$. Then there exists $\displaystyle\lim_{\mathfrak{U}} f(x)$. If a function $g(x)=f(x)$ almost everywhere on $\R^n$, then there exists $\displaystyle\lim_{\mathfrak{U}} g(x)=\lim_{\mathfrak{U}} f(\xi)$.
\end{lemma}
\begin{proof}
By Corollary~\ref{cor3} \ $\mathfrak{U}$ is an ultrafilter. By the known properties of ultrafilters, the image $f_*\mathfrak{U}$ is an ultrafilter on the compact $[-M,M]$, where $M=\sup|f(x)|$, and this ultrafilter converges to some point $y\in [-M,M]$. Therefore, $\displaystyle\lim_{\mathfrak{U}} f(x)=\lim f_*\mathfrak{U}=y$. Further, suppose that a function $g=f$ a.e. on $\R^n$. Then the set $E=\{ x\in\R^n \ | \ g(x)\not=f(x) \ \}$ has null Lebesgue measure. Let $V$ be a neighborhood of $y$. Then $g^{-1}(V)\supset f^{-1}(V)\setminus E$. By the convergence of the ultrafilter $f_*\mathfrak{U}$ the set $f^{-1}(V)\in\mathfrak{U}$. Since $\mathfrak{U}$ is an essential ultrafilter while $f^{-1}(V)\setminus E\sim f^{-1}(V)$, then $f^{-1}(V)\setminus E\in\mathfrak{U}$. This set is contained in $g^{-1}(V)$, and we claim that $g^{-1}(V)\in\mathfrak{U}$. Since $V$ is an arbitrary neighborhood of $y$, we conclude that $\displaystyle\lim_{\mathfrak{U}} g(x)=y$. The proof is complete.
\end{proof}

By the statement of Lemma~\ref{lem4}, the functional $\displaystyle f\to\lim_{\mathfrak{U}} f(\xi)$ is well-defined on $L^\infty(\R^n)$ and it is a linear multiplicative functional on $L^\infty(\R^n)$. In other words, this functional belongs to the spectrum $\S$ of algebra $L^\infty(\R^n)$. Let us demonstrate that, conversely, any linear multiplicative functional on $L^\infty(\R^n)$ coincides with the limit along some essential ultrafilter.

\begin{theorem}\label{th8} For each $X\in\S$ there exists an essential ultrafilter $\mathfrak{U}$ such that
\begin{equation}\label{23}
\<X,f\>=\lim_{\mathfrak{U}} f(x) \quad \forall f\in L^\infty(\R^n).
\end{equation}
\end{theorem}

\begin{proof}
We denote by $\chi_B=\chi_B(x)$ the indicator function of measurable set $B\subset\R^n$, and define
$$
\mathfrak{F}=\{ \ A\subset\R^n \ | \  \<X,\chi_B\>=1 \ \mbox{ for some measurable } B\subset A \ \}.
$$
It is directly verified that $\mathfrak{F}$ is an essential filter. Let us show that for every $f(x)\in L^\infty(\R^n)$ there exists
$\displaystyle\lim_{\mathfrak{F}} f(x)$. Let $\lambda=\<X,f\>$, $\varepsilon>0$,
$$V=V_\varepsilon=\{ \ x\in\R^n \ | \ |f(x)-\lambda|<\varepsilon \ \},$$
$\overline{V}=\R^n\setminus V$. It is clear that $V$ is a measurable set. We are going to prove that $\<X,\chi_V\>=1$. We define the function
$$
g(x)=\left\{\begin{array}{ccr} 1/(f(x)-\lambda) & , & x\in\overline{V}, \\ 0 & , & x\in V. \end{array}\right.
$$
Since $|f(x)-\lambda|\ge\varepsilon$ on the set $\overline{V}$, then $g(x)\in L^\infty(\R^n)$ and, evidently, ${g(x)(f(x)-\lambda)=\chi_{\overline{V}}}$. Therefore,
$$
\<X,\chi_{\overline{V}}\>=\<X,g\>(\<X,f\>-\lambda)=0.
$$
This implies that
$$
\<X,\chi_V\>=\<X,1-\chi_{\overline{V}}\>=1-\<X,\chi_{\overline{V}}\>=1,
$$
as was to be proved. Hence, $V=V_\varepsilon\in\mathfrak{F}$ for all $\varepsilon>0$, which means that ${\displaystyle\lim_{\mathfrak{F}} f(x)=\lambda=\<X,f\>}$. Notice that the latter relation holds for every $f\in L^\infty(\R^n)$.
Let $\mathfrak{U}$ be an essential ultrafilter such that $\mathfrak{F}\subset\mathfrak{U}$. Then relation (\ref{23}) is fulfilled.
\end{proof}

Notice that the essential ultrafilter indicated in Theorem~\ref{th8} is not unique, but it belongs to a unique equivalence class corresponding to the relation
\begin{equation}\label{24}
\mathfrak{U_1}\sim\mathfrak{U_2} \ \Leftrightarrow \ \lim_{\mathfrak{U_1}} f=\lim_{\mathfrak{U_2}} f \ \forall f\in L^\infty(\R^n)
\end{equation}
on the set of essential ultrafilters.

By Theorem~\ref{th8} any generalized characteristic $X(t)=x_t(X_0)$ can be described as a curve $\mathfrak{U}(t)$ on a set of essential ultrafilters

We call an ultrafilter $\mathfrak{U}$ bounded if it contains a bounded set. It is clear that a bounded ultrafilter $\mathfrak{U}$ contains some compact set $K$. Then
$\mathfrak{U}|_{K}=\{ \ B\in\mathfrak{U} \ | \ B\subset K \ \}$ is an ultrafilter on the compact $K$ and, therefore, it converges to some element $y\in K$. Then
$y=\lim\mathfrak{U}$. We have established that any bounded ultrafilter on $\R^n$ converges. Notice that, conversely,
if an ultrafilter $\mathfrak{U}$ converges, $y=\lim\mathfrak{U}$, then $\mathfrak{U}$ contains all neighborhoods of $y$ and, therefore, is bounded.

By Theorem~\ref{th8} any generalized characteristic $X(t)=x_t(X_0)$ can be described as a curve $\mathfrak{U}(t)$, $t\in\R$ on a set of essential ultrafilters, which is uniquely defined up to the equivalence (\ref{24}). We complete this section by the following result.

\begin{theorem}\label{th9}
Let $\mathfrak{U}(t)$, $t\in\R$, be a generalized characteristic. Assume that the essential ultrafilter $\mathfrak{U}(t_0)$ is bounded for some $t_0\in\R$. Then $\mathfrak{U}(t)$ is bounded for all $t\in\R$, and the curve $x(t)=\lim \mathfrak{U}(t)$, $t\in\R$, is Lipschitz: $|x(t)-x(t_0)|\le N|t-t_0|$.
\end{theorem}

\begin{proof}
Since the ultrafilter $\mathfrak{U}(t_0)$ is bounded, there exists the limit $x(t_0)=\lim\mathfrak{U}(t_0)$. Therefore, for every $\varepsilon>0$ the ball $$V_\varepsilon=\{ \ x\in\R^n \ | \ |x-x(t_0)|<\varepsilon \ \} \in\mathfrak{U}(t_0).$$ Denote by $u_0(x)$ the indicator function of this ball and let $u(t,x)\in L^\infty(\R^{n+1})$ be the unique g.s. of equation (\ref{1}) satisfying the Cauchy condition $u(t_0,x)=u_0(x)$. As readily follows from the statements of Proposition~\ref{pro1}, $u(t,x)=0$ for $|x-x(t_0)|\ge \varepsilon+N|t-t_0|$. By the definition of generalized characteristics
\begin{equation}\label{25}
u(t,x)=\lim_{\mathfrak{U}(t)} u(t,\cdot)=\lim_{\mathfrak{U}(t_0)} u_0=1.
\end{equation}
Let us show that the ball $$V_{\varepsilon+N|t-t_0|}=\{ \ x\in\R^n \ | \ |x-x(t_0)|<\varepsilon+N|t-t_0| \ \}\in\mathfrak{U}(t).$$ Otherwise, its complement $\overline{V_{\varepsilon+N|t-t_0|}}\in\mathfrak{U}(t)$.
Since $u(t,x)=0$ on this set, we claim that $\lim\limits_{\mathfrak{U}(t)} u(t,\cdot)=0$.
This contradicts (\ref{25}), therefore, we conclude that $V_{\varepsilon+N|t-t_0|}\in\mathfrak{U}(t)$. Hence,
the ultrafilter $\mathfrak{U}(t)$ is bounded and $x(t)\doteq\lim\mathfrak{U}(t)$ lays in the closure of $V_{\varepsilon+N|t-t_0|}$. This implies that $|x(t)-x(t_0)|\le \varepsilon+N|t-t_0|$. Since $\varepsilon>0$ is arbitrary, we conclude that $|x(t)-x(t_0)|\le N|t-t_0|$.
\end{proof}

Remark that the curves $x=x(t)=\lim\mathfrak{U}(t)$, $t\in\R$ can be treated as the projection of a generalized characteristic $\mathfrak{U}(t)$ on the ``physical'' space $\R^n$. In some sense $x(t)$ can be interpreted as a solution of characteristic system (\ref{10}). As opposed to classic solutions, $x(t)$ is not uniquely determined by
$(t_0,x(t_0))$, actually it is determined by a point $(t_0,\mathfrak{U}(t_0))$.

{\bf Acknowledgement.}
This research was carried out with the financial support of the Russian Foundation for Basic Research (grant no. 15-01-07650-a) and the Ministry of Education and Science of Russian Federation, project no. 1.857.2014/K in the framework of state task.

\end{document}